\documentclass[12pt,a4paper,reqno]{amsart}
\usepackage{amssymb}
\usepackage{amscd}
\usepackage{enumerate}
\usepackage{graphicx}
\usepackage{siunitx}
\usepackage{tikz-cd}
\usepackage{color}
\usepackage{MnSymbol}
\usetikzlibrary{arrows}
\numberwithin{equation}{section}

\usepackage{mathtools}
\usepackage[tableposition=top]{caption}
\usepackage{booktabs,dcolumn}

\DeclareFontFamily{OT1}{rsfs}{}
\DeclareFontShape{OT1}{rsfs}{n}{it}{<-> rsfs10}{}
\DeclareMathAlphabet{\mathscr}{OT1}{rsfs}{n}{it}

\theoremstyle{plain}

\newtheorem{theorem}{Theorem}[section]
\newtheorem{proposition}[theorem]{Proposition}
\newtheorem{lemma}[theorem]{Lemma}

\newtheorem{problem}[theorem]{Problem}
\theoremstyle{definition}

\newcommand\E{\mathbf{E}}
\renewcommand\P{\mathbf{P}}
\newcommand\Var{\mathbf{Var}}
\newcommand\R{\mathbb{R}}
\newcommand\Z{\mathbb{Z}}
\newcommand\N{\mathbb{N}}

\newcommand\C{\mathbb{C}}

\newcommand\eps{\varepsilon}

\begin{document}

\title{Exploring the toolkit of Jean Bourgain}

\author{Terence Tao}
\address{UCLA Department of Mathematics, Los Angeles, CA 90095-1555.}
\email{tao@math.ucla.edu}



\begin{abstract}  Gian-Carlo Rota asserted in \cite{rota} that ``every mathematician only has a few tricks''.  The sheer breadth and ingenuity in the work of Jean Bourgain may at first glance appear to be a counterexample to this maxim.  However, as we hope to illustrate in this article, even Bourgain relied frequently on a core set of tools, which formed the base from which problems in many disparate mathematical fields could then be attacked.  We discuss a selected number of these tools here, and then perform a case study of how an argument in one of Bourgain's papers \cite{similarity} can be interpreted as a sequential application of several of these tools.
\end{abstract}

\maketitle


\section{Introduction}

As the other articles in this collection demonstrate, Jean Bourgain achieved breakthroughs in an astounding number of areas across mathematics, many of which would seem at first glance to be quite unrelated to each other.  It is almost beyond belief that a single mathematician could have such deep impact on so many different subfields during his or her career.  However, if one compares Bourgain's works in different topics, some common threads begin to emerge.  In particular, one discovers that Bourgain had a certain ``toolkit'' of basic techniques that he was extremely skilled at using in a great variety of situations.  While this toolkit is far from sufficient to explain the powerful nature of his work, it does show how he could at least get started on making progress in so many different mathematical problems.  In this article we present a selection of Bourgain's most commonly used tools: quantification of qualitative estimates, dyadic pigeonholing, random translations, and metric entropy and concentration of measure.  These tools often did not originate with Bourgain's work, but he was able to wield them systematically and apply them to a far broader range of problems than had previously been realized.  This is far from a complete list of Bourgain's standard tools - for instance, many of his works also feature systematic use of the uncertainty principle, or the method of probabilistic decoupling - but it is an illustrative sample of that basic toolkit.

Of course, in many cases Bourgain's work involved much deeper and delicate arguments than just the general-purpose techniques presented here.  Nevertheless, knowledge of these basic tools helps place Bourgain's arguments in a more systematic framework, in which these preliminary techniques are used to isolate the core features of the problem, which were then attacked by the full force of Bourgain's intellectual firepower.  But sometimes these basic methods are already enough to solve non-trivial problems almost on their own.  We illustrate this in the final section (Section \ref{case}) by tracing through a single paper \cite{similarity} of Bourgain's, giving one of the most general results known on the Erd\H{o}s similarity problem; as we hope to demonstrate, this non-trivial result can be interpreted as essentially being a sequential application of each of the tools listed here in turn.

\section{Notation}

We use $1_E$ to denote the indicator of a set $E$, and use $|E|$ to denote either the cardinality of $E$ (if it is finite) or its Lebesgue measure (if it is infinite).  If $E$ lies in a vector space, we use $x+E \coloneqq \{ x+y: y \in E \}$ for the translate of $E$ by $x$, $x - E \coloneqq  \{ x-y: y \in E \}$ for the reflection of $E$ across $x/2$, and $\lambda \cdot E \coloneqq  \{ \lambda y: y \in E \}$ for the dilation of $E$ by $\lambda$.

We use $\P(E)$ to denote the probability of a random event $E$, and $\E X$ and $\Var(X)$ for the mean and variance of a random variable $X$.

We will use the asymptotic notation $X \lesssim Y$, $Y \gtrsim X$, or $X = O(Y)$ to denote the bound $|X| \leq CY$ for an absolute constant $C$. If the constant $C$ depends on parameters we will indicate this by parameters, for instance $X \lesssim_{p,d} Y$ denotes the bound $|X| \leq C_{p,d} Y$ for a constant $C_{p,d}$ depending only on $p,d$.  It may be noted that in Bourgain's work the implied constant was sometimes omitted completely from the notation.  In the words of Heath-Brown, in his Mathematical Reviews summary of Bourgain's paper \cite{dirichlet},
``Readers are advised not to take some of the assertions made in the course of the proofs too literally, and, in particular, to abandon any preconceived notion as to the meaning of the symbol ``$\sim$''''.

\section{Quantitative formulation of qualitative problems}\label{quant-sec}

One can roughly divide analysis into ``soft analysis'' - the study of qualitative properties (continuity, measurability, integrability, etc.) of infinitary objects, and ``hard analysis'' - the study of quantitative estimation of finitary objects.  Bourgain's toolkit lies almost exclusively in the latter category, so when tackling a ``soft analysis'' problem, often the first step in one of Bourgain's arguments is to locate a more quantitative ``hard analysis'' estimate that will imply the desired claim, removing almost all the appearances of limits or arbitrarily large and small scales, and instead working with a large but finite number of scales and focusing on estimates that are uniform with respect to several parameters.

For instance, consider the following result of Furstenberg, Katznelson, and Weiss \cite{fkw}:

\begin{theorem}[Furstenberg-Katznelson-Weiss theorem, qualitative version]\label{fkw}  Let $A \subset \R^2$ be a measurable set whose upper density $\delta \coloneqq \limsup_{R \to \infty} \frac{|A \cap B(0,R)|}{|B(0,R)|}$ is positive.  Then there exists $l_0$ such that for all $l \geq l_0$, there exist $x,y \in A$ with $|x-y| \geq l$.
\end{theorem}

Note that this theorem does not provide any quantitative bound for the length threshold $l_0$ in terms of the upper density $\delta$.  Indeed, such a bound is not possible, since if one replaces $A$ by a rescaled version $\lambda \cdot A \coloneqq \{ \lambda x: x \in A \}$ then the length threshold $l_0$ will be replaced by $\lambda l_0$, while the upper density $\delta$ remains unchanged.  As such, one may be tempted to conclude that Theorem \ref{fkw} is irredeemably ``qualitative'' in nature.  Nevertheless, in \cite{density}, Bourgain gave a new proof of this theorem (as well as several novel generalisations) by first establishing the following quantitative analogue:

\begin{theorem}[Furstenberg-Katznelson-Weiss theorem, quantitative version]\label{borg-fkw}  Let $0 < \eps < \frac{1}{2}$, let $B \subset [-1,1]^2$ have measure $|B| \geq \eps$, and let $J = J(\eps)$ a be sufficiently large natural number depending on $\eps$.  Suppose that $0 < t_J < \dots < t_1 \leq 1$ are a sequence of scales with $t_{j+1} \leq t_j/2$ for all $1 \leq j < J$.  Then for at least one $1 \leq j \leq J$, one has
\begin{equation}\label{ar}
 \int_{\R^2} \int_{S^1} 1_B(x) 1_B(x+t_j\omega)\ d\sigma(\omega) dx \gtrsim \eps^2,
\end{equation}
where $d\sigma$ is surface measure on the unit circle $S^1$, normalised to have unit mass.
\end{theorem}

Note that there is no longer any appearance of limits in Theorem \ref{borg-fkw}, and instead of working with an infinity of scales $l$, one now works with a finite number of scales $t_1,\dots,t_J$.  Furthermore, there is a uniform bound on the number $J$ of scales involved depending on $\eps$; the arguments in \cite{density} in fact give the explicit dependence $J = O( \frac{1}{\eps} \log \frac{1}{\eps})$ on $\eps$.

Let us see why Theorem \ref{borg-fkw} implies Theorem \ref{fkw}.  Assume for contradiction that Theorem \ref{fkw} failed, then one can find a set $A \subset \R^2$ of some positive upper density $\delta>0$, and a sequence of scales $l_1 < l_2 < \dots$ going to infinity such that for each $l_j$ there are no $x,y \in A$ with $|x-y|=l_j$.  We can sparsify this sequence of scales so that $l_{j+1} \geq 2l_j$ for all $j$.  Let $\eps>0$ be sufficiently small (depending only on $\delta$), and let $J$ be as in Theorem \ref{borg-fkw}.  As $A$ has density $\delta$, one can find a radius $R > l_J$ such that $|A \cap B(0,R)| \gtrsim \delta R^2$.  If we now consider the rescaling $B \coloneqq \{ x \in [-1,1]^2: Rx \in A \}$, and define the rescaled scales $0 < t_J < \dots < t_1 \leq 1$ by $t_j \coloneqq  l_{J+1-j}/R$ for $j=1,\dots,R$, then $|B| \gtrsim \delta$ and $t_{j+1} \leq t_j/2$ for all $1 \leq j < J$, and for any $1 \leq j \leq J$ there are no points $x,y \in A_R$ with $|x-y|=t_j$.  In particular the left-hand side of \eqref{ar} vanishes for all $1 \leq j \leq J$, and one then contradicts Theorem \ref{borg-fkw} after choosing the parameters appropriately.  Note how this argument fails to give any bound on $l_0$ (as it must), despite being ``quantitative'' in nature; it gives quantitative bounds on the \emph{number} of genuinely different scales $l_j$ at which the conclusion of Theorem \ref{fkw} fails, but does not bound the \emph{magnitude} of these scales.  

We will sketch the proof of Theorem \ref{borg-fkw} in the next section.  For now, we turn to another example of Bourgain's strategy of attacking qualitative results through quantitative methods.  We focus on a specific pointwise ergodic theorem established in \cite[Theorem 1]{ergodic}:

\begin{theorem}[Bourgain's pointwise ergodic theorem along squares]\label{pet}  Let $(X, \mu)$ be a probability space with a measure-preserving transformation $T: X \to X$.  Let $f \in L^r(X,\mu)$ for some $r>1$.  Then the averages
$$ A_N f(x) \coloneqq   \frac{1}{N} \sum_{n=1}^N f( T^{n^2} x )$$
converge pointwise as $N \to \infty$ for $\mu$-almost every $x \in X$.
\end{theorem}

In fact one can replace the squares here by any other polynomial with integer coefficients, but we focus on the squares for sake of concreteness.
A standard method to establish almost everywhere convergence results in ergodic theory is to establish an inequality for the associated maximal function
\begin{equation}\label{max}
 Mf(x) \coloneqq \sup_{N>0} |A_N f(x)|,
\end{equation}
and combine this with an almost everywhere convergence result for functions $f$ in a dense subclass of the original space $L^r(X,\mu)$ of interest (e.g., $L^\infty(X,\mu)$).  In \cite{ergodic} a maximal inequality for \eqref{max} was established, thus reducing matters to consideration of functions $f$ in the dense subclass $L^\infty(X,\mu)$, but new ideas were needed to handle this subclass.  Bourgain achieved this through the following variant of a maximal inequality:

\begin{theorem}[Bourgain's variational estimate]\label{bve}  Let $\lambda>1$, let $\eps>0$, and let $J$ be sufficiently large depending on $\lambda,\eps$.  Then for any $1 < N_1 < N_2 < \dots < N_J$, any $(X,\mu,T)$ as in Theorem \ref{pet}.  Then for any $f \in L^\infty(X,\mu)$, one has
$$ \sum_{j=1}^{J-1} \left\| \sup_{N_j \leq N \leq N_{j+1}: N \in Z_\lambda} |A_N f - A_{N_j} f| \right\|_{L^2(X,\mu)} \lesssim \eps J \|f\|_{L^2(X,\mu)}$$
where $Z_\lambda \coloneqq  \{ \lfloor \lambda^n\rfloor: n \in \N \}$.
\end{theorem}

Note that $L^2$ boundedness of the maximal function \eqref{max} would suffice to establish Theorem \ref{bve} were it not for the additional factor of $\eps$ on the right-hand side; thus 
We now sketch how Theorem \ref{bve} implies Theorem \ref{pet}.  By the previous discussion we may assume $f \in L^\infty(X,\mu)$, and may normalise $\|f\|_{L^\infty(X,\mu)} = 1$ and take $f$ to be real-valued.  By rounding a natural number $N$ to the nearest element $N'$ of $Z_\lambda$ for some $\lambda>1$ we see that $A_N f(x) = A_{N'} f(x) + O( \lambda-1)$.  From this it is not difficult to see that we only need to establish pointwise almost everywhere convergence of $A_N f(x), N \in Z_\lambda$ for any fixed $\lambda>1$.  Suppose this claim failed, then there would be a set $E$ of positive measure in $X$ and a $\delta>0$ such that
$$ \limsup_{N \in Z_\lambda} A_N f(x) - \liminf_{N \in Z_\lambda} A_N f(x) \geq \delta$$
for all $x \in E$.  By standard measure theory arguments one can then recursively construct an infinite sequence $N_1 < N_2 < \dots$ and a subset $E'$ of $E$ of positive measure such that
$$ \sup_{N_j \leq N \leq N_{j+1}: N \in Z_\lambda} |A_N f(x) - A_{N_j} f(x)| \gtrsim \delta$$
for all $j \geq 1$ and $x \in E'$.  This then can be used to contradict Theorem \ref{bve} after selecting $\eps$ small enough and $J$ large enough.

The bulk of the paper \cite{ergodic} is occupied with the task of establishing estimates such as Theorem \ref{bve}, which proceeds by transferring the problem to the integers, applying Fourier-analytic decompositions, and establishing some maximal inequalities of a harmonic analysis flavor.  These arguments have proven to be quite influential, and a paradigm for establishing many other pointwise ergodic theorems, but we will not survey these developments here.

\section{Dyadic pigeonholing}\label{dyadic-sec}

One of the oldest tricks in analysis is that of \emph{dyadic decomposition}: when faced with a sum or integral over a parameter ranging over a wide range of scales, first control the contribution of an individual dyadic scale (such as when the magnitude of the parameter ranges between two fixed consecutive powers $2^k, 2^{k+1}$ of two), and then sum over all dyadic scales.  For instance, we have the Cauchy condensation test: when asked to determine whether a series $\sum_{n=1}^\infty f(n)$ is absolutely convergent, where $f: \N \to \R^+$ is non-negative and non-increasing, one can break up the sum dyadically
$$ \sum_{n=1}^\infty f(n) = \sum_{k=0}^\infty \sum_{2^k \leq n < 2^{k+1}} f(n)$$
and then observe that each dyadic component can be easily bounded above and below
$$ 2^k f(2^{k+1}) \leq \sum_{2^k \leq n < 2^{k+1}} f(n) \leq 2^k f(2^k)$$
at which point one easily sees that the original series $\sum_{n=1}^\infty f(n)$ converges if and only if the condensed sum $\sum_{k=0}^\infty 2^k f(2^k)$ converges.  While both sums are infinite, in practice the latter sum is significantly more tractable than the former; for instance any polynomial improvements $n^{-\eps}$ to bounds for the original sequence $f(n)$ leads to exponential improvements $2^{-\eps k}$ in the bounds for the new sequence $2^k f(2^k)$.  

A surprisingly useful variant of this method was used repeatedly by Bourgain in many problems, in which dyadic decomposition is combined with the pigeonhole principle to locate a single ``good'' scale in which to run additional arguments.  We refer to this combination of dyadic decomposition and the pigeonhole principle as \emph{dyadic pigeonholing}.  

The quantitative result claimed in Theorem \ref{borg-fkw} is already well suited to a dyadic pigeonholing argument, since the scales $t_j$ in that argument are already at least dyadically separated, and we can sketch its proof as follows.  Using the Fourier transform $\hat f(\xi) \coloneqq \int_{\R^d} f(x) e^{2\pi i x \cdot \xi}\ dx$, one can rewrite the left-hand side of \eqref{ar} as
$$ \int_{\R^2} |\hat 1_B(\xi)|^2 \hat \sigma( t_j \xi)\ d\xi,$$
where $\hat \sigma(\xi) \coloneqq \int_{S^1} e^{2\pi i \omega \cdot \xi}\ d\sigma(\omega)$ is the Fourier transform of the surface measure $d\sigma$.  One can split this integral into the contribution of the ``low frequencies''
\begin{equation}\label{low-fourier}
\int_{|\xi| \leq \delta / t_j} |\hat 1_B(\xi)|^2 \hat \sigma( t_j \xi)\ d\xi,
\end{equation}
the ``medium frequencies''
\begin{equation}\label{medium-fourier}
\int_{\delta/t_j \leq |\xi| \leq 1 / \delta t_j} |\hat 1_B(\xi)|^2 \hat \sigma( t_j \xi)\ d\xi,
\end{equation}
and the ``high frequencies''
\begin{equation}\label{high-fourier}
\int_{|\xi| > 1 / \delta t_j} |\hat 1_B(\xi)|^2 \hat \sigma( t_j \xi)\ d\xi,
\end{equation}
where $0 < \delta = \delta(\eps) \leq 1/2$ is a small quantity depending on $\eps$ to be chosen later.
For the contribution \eqref{low-fourier} of the low frequencies, the factor $\hat \sigma(t_j \xi)$ is close to $1$, and it is not difficult to obtain a lower bound on this quantity that is $\gtrsim |B|^2 \geq \eps^2$ if $\delta$ is small enough.  For the contribution \eqref{high-fourier} of the high frequencies, the factor $\hat \sigma(t_j \xi)$ is quite small, and one can show that the contribution of this term is negligible, again for $\delta$ small enough.  The problematic term is the contribution \eqref{medium-fourier}, which one can of course upper bound by
$$ \int_{\delta/t_j \leq |\xi| \leq 1 / \delta t_j} |\hat 1_B(\xi)|^2\ d\xi.$$
By Plancherel's theorem one can upper bound this crudely by
$$ \lesssim \int_{\R^2} |\hat 1_B(\xi)|^2\ d\xi = |B|,$$
but this bound is too weak compared to the lower bound of $\gtrsim |B|^2$ one can obtain for the main term \eqref{low-fourier}.  However, note that because of the lacunarity hypothesis $t_{j+1} \leq t_j/2$, the annuli $\{ \delta/t_j \leq |\xi| \leq 1 / \delta t_j\}$ only overlap with multiplicity $O(\log \frac{1}{\delta})$, hence the Plancherel bound actually gives
\begin{equation}\label{sumj}
 \sum_{j=1}^J \int_{\delta/t_j \leq |\xi| \leq 1 / \delta t_j} |\hat 1_B(\xi)|^2\ d\xi \lesssim \log \frac{1}{\delta} |B|
\end{equation}
and hence by the pigeonhole principle we can find a scale $t_j$ for which
$$ \sum_{j=1}^J \int_{\delta/t_j \leq |\xi| \leq 1 / \delta t_j} |\hat 1_B(\xi)|^2\ d\xi \lesssim \frac{\log \frac{1}{\delta}}{J} |B|.$$
For $J$ large enough, one can use this ``good'' scale to make the contribution \eqref{medium-fourier} of the medium frequencies small compared to that of the high frequencies, and one can conclude the proof of Theorem \ref{borg-fkw}.

Another typical instance of dyadic pigeonholing occurs in \cite[Lemma 2.15]{besicovitch} when Bourgain establishes new lower bounds on the Hausdorff dimension of Besicovitch sets $E$ (compact subsets of $\R^d$ that contain a unit line segment $\ell_\omega$ in every direction).  To lower bound the Hausdorff such a set by $\alpha$, one would have to establish a lower bound for the Hausdorff content $\sum_i r_i^{\alpha-\eps}$ whenever one covers the set $E$ by small balls $B(x_i,r_i)$.  By rounding up each $r_i$, we can assume without loss of generality that each $r_i$ is a power of two (dyadic pigeonholing), and replace balls by cubes; we can then group together the cubes of a given size and obtain a covering
$$ E \subset \sum_{j \geq j_0} B_j$$
where $j_0$ is large and $B_j$ is a union of cubes of sidelength $2^{-j}$.  To get the required lower bound it would then suffice to show that for at least one of the scales $j$, the number of cubes used to form $B_j$ is $\gtrsim 2^{j(\alpha-\eps)}$.  But which scale $j$ to use?  Since the $B_j$ cover each $\ell_\omega$, we have
$$ \sum_{j \geq j_0} {\mathcal H}^1( \ell_\omega \cap B_j ) \geq 1$$
for each of the unit line segments $\ell_\omega$, hence on integrating over all directions $\omega \in S^{d-1}$ using Fubini's theorem we have
$$ \sum_{j \geq j_0} \int_{S^{d-1}} {\mathcal H}^1( \ell_\omega \cap B_j )\ d\omega \gtrsim 1.$$
By the pigeonhole principle, we can then find a scale $j \geq j_0$ for which
$$ \int_{S^{d-1}} {\mathcal H}^1( \ell_\omega \cap B_j )\ d\omega \gtrsim \frac{1}{j^2}$$
(say).  The quantity $1/j^2$ is quite ``large'' compared to the scale $2^{-j}$, and this estimate asserts (roughly speaking) that ``many'' of the $\ell_\omega$ have ``large'' intersection with the $B_j$.  Having selected such a good scale $\delta = 2^{-j}$, Bourgain was able to proceed to establish new bounds on $\alpha$ by estimation of an expression now known as the \emph{Kakeya maximal function} associated to this scale; see the companion paper \cite{demeter} for further discussion of this function and its applications to the Fourier restriction problem.

The dyadic pigeonholing method does not need to explicitly involve powers of two (or other lacunary sequences of scales).  One of Bourgain's earlier uses of the method appears\footnote{We thank Assaf Naor for this reference.} in his work \cite[\S 5]{lipschitz} on quantitative versions of a Lipschitz embedding theorem of Ribe \cite{ribe}, where at one point in the argument he has a Lipschitz function $F: E \to Y$ from a finite-dimensional normed space $E$ to a Banach space $Y$, and wishes to locate a scale $t$ at which the Poisson integral $F * P_t$ of $F$ has a large directional derivative $\partial_a(F*P_t)(x)$ at one point $x \in E$.  This scale $t$ roughly corresponds to the spatial scale $2^{-j}$ in the previous discussion.  To locate a good scale $t$, Bourgain first observes from the triangle inequality that
\begin{equation}\label{fpts}
 \|\partial_a(F*P_t)\| * P_s(x) \geq \| \partial_a(F * P_{t+s}) \|(x)
\end{equation}
and hence the integral quantity $\int_E \|\partial_a(F*P_t)\|(x)\ dx$ is non-increasing in $t$.  On the other hand, the specific construction of the function $F$ in \cite{lipschitz} provided an upper bound on this integral that is uniform in $t$.  Applying the pigeonhole principle, one can then find a scale $t>0$ for which one has
$$\int_E \|\partial_a(F*P_t)\|(x)\ dx \approx \int_E \|\partial_a(F*P_{(R+1)t})\|(x)\ dx$$
for some moderately large parameter $R>0$, and where we shall be vague about the precise meaning of the symbol $\approx$; comparing this with \eqref{fpts} and other properties of $F$ eventually gives the desired lower bound on $\partial_a(F*P_t)(x)$.  See the companion paper \cite{ball} for further discussion of the impact of Bourgain's ``Ribe program''.

Dyadic pigeonholing makes a small but important role in an important result \cite{nls} of Bourgain on the energy-critical nonlinear Schr\"odinger equation, discussed in more detail in Kenig's article \cite{kenig}:

\begin{theorem}[Global regularity for energy-critical NLS for radial data]  Let $u_0 \in H^1(\R^3)$ be smooth and spherically symmetric, then there is a unique smooth finite energy global solution $u: \R \times \R^3 \to \C$ to the nonlinear Schr\"odinger equation $i \partial_t u + \Delta u = |u|^4 u$ with initial data $u(0,x) = u_0(x)$.
\end{theorem}

At one stage \cite[\S 4]{nls} in the (rather intricate) argument, a solution $u$ is constructed to exhibit a concentration property at a certain time $t_{j_s}$ and frequency $N_{j_s}$, in that (suppressing a parameter $\eta$ that is not relevant for the current discussion)
\begin{equation}\label{pnu}
 \| P_{N_{j_s}} u(t_{j_s}) \|_{L^2(\R^3)} \gtrsim N_{j_s}^{-1},
\end{equation}
where $P_{N_{j_s}}$ is a Fourier projection of Littlewood-Paley type to the frequency region $\{ \xi: |\xi| \sim N_{j_s}\}$.  In Bourgain's argument it is necessary to propagate this lower bound from time $t_{j_s}$ to a nearby time $t_{j_r}$.  The NLS equation conserves the full $L^2$ mass $\| u(t) \|_{L^2(\R^3)}^2$, but this cannot be directly used here due to the solution $u$ potentially having an extremely large amount of $L^2$ mass at low frequencies.  To resolve this, Bourgain applies a Fourier truncation operator $P_{\geq N}$ restricting to frequencies $|\xi| \gtrsim N$ for some $N < N_{j_s}$, and exploits the approximately conserved nature of the high-frequency portion $\| P_{\geq N} u(t) \|_{L^2(\R^3)}^2$ of the mass, by a computation of the time derivative
$$ \partial_t \| P_{\geq N} u(t) \|_{L^2(\R^3)}^2.$$
There are several components to this time derivative, but the dominant contribution comes from the portion of the solution $u$ residing at frequency scales $M$ comparable to $N$ (intuitively, this reflects the potential exchange of mass in the frequency domain between frequency modes of magnitude just below $N$, and frequency modes of magnitude just above $N$).  For any given $N$, the upper bounds on this time derivative could overwhelm the lower bound in \eqref{pnu}; however, by obtaining an estimate for the sum over a range of $N$ (in a manner analogous to \eqref{sumj}) and applying the pigeonhole principle, one can locate a ``good scale'' $N$ for which one has satisfactory control on the derivative of the mass.  This idea to use the dyadic pigeonholing method to establish approximate conservation laws has wide application; for instance, it was used by Rodgers and myself recently to resolve the Newman conjecture \cite{newman} in analytic number theory.

In some cases one can use more sophisticated tools than the pigeonhole principle to locate a good scale.  One example of this arises when establishing Bourgain's quantitative refinement \cite{roth} of Roth's theorem \cite{roth-thm}:

\begin{theorem}[Bourgain-Roth theorem]  Let $N \geq 10$, and let $A \subset \{1,\dots,N\}$ be a set containing no three-term arithmetic progressions.  Then $|A| \lesssim \frac{(\log\log N)^{1/2}}{\log^{1/2} N} N$.
\end{theorem}

A key innovation in this paper was to manipulate \emph{Bohr sets}
$$ B(S,\rho) \coloneqq \{ n \in \Z: |n| \leq N; \| n \theta \| \leq \rho \forall \theta \in S \}$$
where $\rho>0$ is a ``radius'', $S$ is a collection of frequencies $\theta \in \R/\Z$, and $\|x\|$ denotes the distance of $x$ to the nearest integer.  These sets generalize the long arithmetic progressions that appear prominently in previous work in this area such as \cite{roth-thm}, and are well adapted to the Fourier-analytic methods used to establish Roth-type theorems.  However, a key difficulty arises due to the discontinuous nature of the Bohr sets in $\rho$; in particular, if $\rho'$ is close to $\rho$, there is no \emph{a priori} reason why the cardinality of $B(S,\rho')$ should be close to that of $B(S,\rho)$.  However, the cardinality $|B(S,\rho)|$ is clearly non-decreasing in $\rho$, and a simple covering argument gives a doubling bound
$$ |B(S,2\rho)| \lesssim O(1)^{|S|} |B(S,\rho)|.$$
In particular, the distributional derivative $\frac{d}{d\rho} \log |B(S,\rho)|$ is a measure of total variation $O( |S| )$ on any dyadic interval $[\rho_0, 2\rho_0]$.  Combining this with the Hardy-Littlewood maximal inequality, one can conclude that every interval $[\rho_0,2\rho_0]$ contains a radius $\rho$ where the Bohr set $B(S,\rho)$ is ``regular'' in the sense that
$$ |B(S,\rho')| = \exp\left( O\left( |S| \frac{|\rho'-\rho|}{|\rho|} \right)\right) |B(S,\rho)|$$
for all $\rho' > 0$; see for instance \cite[Lemma 4.25]{tao-vu} for this version of the construction.  Regular Bohr sets are now a standard tool in modern additive combinatorics.

\section{Random translations}\label{random-sec}

Consider the interval $I = [0,1/N]$ in the unit circle $\R/\Z$ for some large integer $N$.  Then $I$ is much smaller than $\R/\Z$, but we can cover $\R/\Z$ by $1/|I|=N$ translates $I+j/N, j=0,\dots,N-1$ of $\R/\Z$.  Of course, most subsets $E$ of $\R/\Z$ will not have this perfect tiling property.  However, by using random translations of $E$, one can achieve something fairly close to a perfect tiling:

\begin{lemma}[Random translations]\label{translate}  Let $G = (G,\cdot)$ be a compact group (not necessarily abelian) with Haar probability measure $\mu$.  Let $E$ be a measurable subset of $G$, and let $N$ be a natural number.  Then there exist translates $g_1 E, \dots, g_N E$ of $E$ by some shifts $g_1,\dots,g_N \in G$ with
$$ \mu( g_1 E \cup \dots \cup g_N E ) \geq 1 - (1-\mu(E))^N.$$
\end{lemma}

\begin{proof}  We use the probabilistic method. Let $g_1,\dots,g_N$ be drawn independently at random from $G$ using the Haar measure $\mu$.  Then by the Fubini-Tonelli theorem we have
\begin{align*}
 {\mathbf E} \mu( g_1 E \cup \dots \cup g_N E ) &= \int_G {\mathbf E} 1_{g_1 E \cup \dots \cup g_N E}(x)\ d\mu(x) \\
&= \int_G {\mathbf E} (1 - \prod_{i=1}^N 1_{g_i (G \backslash E)}(x))\ d\mu(x) \\
&= \int_G (1 - \prod_{i=1}^N {\mathbf E} 1_{g_i (G \backslash E)}(x))\ d\mu(x) \\
&= \int_G (1 - \prod_{i=1}^N \mu(G \backslash E))\ d\mu(x) \\
&= 1 - (1-\mu(E))^N
\end{align*}
and the claim follows.
\end{proof}

In particular, if $\mu(E) \sim 1/N$, then we can find $N$ translates $g_1 E, \dots, g_N E$ of $E$ whose union has measure $\sim 1$, thus these translates behave as if they are disjoint ``up to constants''.  We observe that the same claim also holds for any homogeneous space $G/H$ of a compact group $G$ (with the attendant Haar probability measure), simply by lifting subsets of that homogeneous space back up to $G$.

Lemma \ref{translate} allows one in many cases to reduce the analysis of ``small'' subsets of a compact group $G$ (or a homogenous space $G/H$ of $G$) to the analysis of ``large'' sets, particularly if the problem in question enjoys some sort of translation symmetry with respect to the group $G$.  This idea was for instance famously exploited by Stein \cite{stein} in his maximum principle equating almost everwhere convergence results for translation-invariant operators with weak-type $(p,p)$ maximal inequalities.  In \cite[\S 6]{besicovitch}, Bourgain noted that these techniques could also be combined with the factorization theory of Pisier, Nikishin, and Maurey \cite{pisier} (which Bourgain had previously used for instance in \cite{zonotopes}), although it has subsequently been realized that the arguments can be formulated without explicit reference to that theory.  Specifically, in the context of restriction estimates for the sphere, Bourgain observed

\begin{proposition}  Suppose that $d \geq 2$ and $1 < p < 2$ is such that one has the restriction estimate
\begin{equation}\label{don}
 \| \hat f \|_{L^1(S^{d-1}, d\sigma)} \lesssim_{p,d} \|f\|_{L^p(\R^d)}
\end{equation}
for all Schwartz functions $f \colon \R^d \to \C$, where $\sigma$ is normalized surface measure on the sphere $S^{d-1}$.  Then one can automatically improve this to the stronger estimate
$$ \| \hat f \|_{L^{p,\infty}(S^{d-1}, d\sigma)} \lesssim_{p,d} \|f\|_{L^p(\R^d)}.$$
\end{proposition}

\begin{proof}  (Sketch) We can normalize $\|f\|_{L^p(\R^d)}=1$.  Let $\lambda > 0$, and let $E \subset S^{d-1}$ denote the level set
$$ E \coloneqq  \{ \omega \in S^{d-1}: |\hat f(\omega)| \geq \lambda \}.$$
Our task is to show that $\sigma(E) \lesssim_{p,d} \lambda^{-p}$.  A direct application of \eqref{don} only gives the estimate $\sigma(E) \lesssim_{p,d} \lambda^{-1}$, which is inferior when $\lambda$ is large.  However, if we let $N$ be an integer with $N \sim 1/\sigma(E)$, then by Lemma \ref{translate} (applied to the homogeneous space $S^{d-1} \equiv SO(d)/SO(d-1)$) one can find rotations $R_1 E, \dots, R_N E$ of $E$ with $\sigma(\bigcup_{i=1}^N R_i E) \sim 1$.  If one then considers the random sum
$$ F(x) \coloneqq  \sum_{i=1}^N \epsilon_i F(R_i x)$$
where the $\epsilon_i$ are independent random signs $\{-1,+1\}$ (or random gaussian variables), a routine application of Khintchine's inequality reveals that with positive probability, one has
$$ \|F\|_{L^p(\R^d)} \lesssim_p N^{1/p}$$
and
$$ \sigma(\{ \omega \in S^{d-1}: |\hat F(\omega)| \gtrsim \lambda \}) \gtrsim \sigma(\bigcup_{i=1}^N R_i E) \gtrsim 1$$
and hence by \eqref{don}
$$ \lambda \lesssim_{p,d} N^{1/p}$$
which gives the required estimate $\sigma(E) \lesssim_{p,d} \lambda^{-p}$.
\end{proof}

Variations of this argument also appear at several other locations in \cite{besicovitch}.

\section{Metric entropy and concentration of measure}\label{entropy-sec}

If $X$ is a random variable (which for sake of discussion we take to be real-valued) with finite second moment, so that the mean $\E X$ and the variance $\Var(X)$ are both finite, then Chebyshev's inequality asserts that
$$ \P( |X - \E X| \geq \lambda \sqrt{\Var(X)} ) \leq \frac{1}{\lambda^2}$$
for any $\lambda>0$, thus there the random variable $X$ exhibits some \emph{concentration of measure} to the interval $[\E X - \sqrt{\Var(X)}, \E X + \sqrt{\Var(X)}]$, in the sense that the probability of lying far outside this interval drops at a polynomial rate to the (normalized) distance to this interval).  In many situations (particularly if $X$ is somehow ``influenced'' by many ``independent sources of randomness''), the decay is in fact far stronger than the $1/\lambda^2$ decay; exponential or even Gaussian type decay can often be obtained.  For instance, if $X$ is a Gaussian variable, then we have
\begin{equation}\label{xex}
 \P( |X - \E X| \geq \lambda \sqrt{\Var(X)} ) \lesssim \exp( - c\lambda^2)
\end{equation}
for all $\lambda>0$ and some absolute constant $c>0$.  Or, if $X = \sum_{i=1}^n X_i$ is the sum of independent random variables $X_i$ that each lie in some interval $[a_i,b_i]$, then the classical Hoeffding inequality gives
$$ \P\left( |X - \E X| \geq \lambda \sqrt{\sum_{i=1}^n (b_i-a_i)^2} \right) \lesssim \exp( - c\lambda^2).$$
Many further \emph{concentration of measure} inequalities of this type are available; see for instance the text \cite{ledoux} for a systematic discussion.

One can combine these sorts of concentration of measure inequalities to control the large deviations of suprema $\sup_{t \in T} X_t$ of a random process $(X_t)_{t \in T}$, where $t$ is a parameter ranging in some index set $T$.  If for instance $T$ is finite, then from the union bound we have
$$ \P( \sup_{t \in T} X_t > \lambda ) = \P( \bigvee_{t \in T}(X_t > \lambda)) \leq \sum_{t \in T} \P(X_t > \lambda)$$
for any $\lambda>0$.  For $\lambda$ large, one can hope to use concentration of measure inequalities to obtain exponentially strong bounds on each individual probability $\P(X_t > \lambda)$; if $T$ is not too large (e.g., subexponential in cardinality) then this method can lead to non-trivial large deviation bounds on the supremum $\sup_{t \in T} X_t$.  However, often in many cases of interest $T$ is too large for this method to be directly used; for instance, $t$ could be a continuous parameter, in which case $T$ is likely to be uncountably infinite.  But in many applications $T$ has the structure of a totally bounded metric space $(T,d)$, thus for each scale $\eps>0$ there exists some finite subset $T_\eps$ of $T$ with the property that every element of $T$ lies within $\eps$ of some element of $T_\eps$.  The minimal possible cardinality of $T_\eps$ is known as the \emph{metric entropy} of $T$ at scale $\eps$ and we will denote it by $N(T, d;\eps)$.  Applying these nets for each $\eps = 2^{-n}$, one can assign to each element $t \in T$ a chain $t_0,t_1,\dots$ converging to $t$ with $t_n \in T_{2^{-n}}$ and $d(t_n,t_{n+1}) \leq 2^{-n+1}$ for all $n$.  If $X_t$ depends continuously on $t$, the triangle inequality then gives the bound
$$ X_t \leq |X_{t_0}| + \sum_{n=0}^\infty |X_{t_n}-X_{t_{n+1}}|$$
and hence
\begin{equation}\label{chain}
 \sup_{t \in T} X_t \leq \sup_{t \in T_1} |X_{t}| + \sum_{n=0}^\infty \sup_{t \in T_{2^{-n}}; t' \in T_{2^{-n-1}}; d(t,t') \leq 2^{-n+1}} |X_{t}-X_{t'}|.
\end{equation}
It is then possible to obtain good large deviation bounds on the uncountable supremum $\sup_{t \in T} X_t$ by using the previous strategy to obtain large deviation bounds on the finite suprema $\sup_{t \in T_1} |X_{t}|$ and $\sup_{t \in T_{2^{-n}}; t' \in T_{2^{-n-1}}; d(t,t') \leq 2^{-n+1}} |X_{t}-X_{t'}|$, which one then combines using crude tools such as the union bound.  We refer to this as the \emph{chaining argument}; it is particularly effective when one has good bounds on the metric entropies $N(T,d;2^{-n})$.  A prototype application of the chaining argument is \emph{Dudley's inequality} \cite{dudley}
$$ \E \sup_{t \in T} X_t \lesssim \sum_{n \in \Z} 2^{-n} \sqrt{\log N(T, d; 2^{-n})}$$
whenever $(X_t)_{t \in T}$ is a (mean zero) Gaussian process with $T$ equipped with the metric (or more precisely, pseudo-metric)
$$ d(t, t') \coloneqq \sqrt{ \Var(X_t-X_{t'})}.$$
This inequality can be readily proven by combining \eqref{chain} with \eqref{xex} and the union bound; see for instance \cite[\S 2]{major} for a clear treatment.  However, the chaining argument is substantially more general and flexible than this, with many early applications of the chaining method due to Bourgain.  Perhaps the most well known is the following result:

\begin{theorem}[Random sets of orthonormal systems have the $\Lambda(p)$ property]\label{lambda-p} Let $\phi_1,\dots,\phi_n$ be a system of bounded orthonormal functions on a probability space $(X,\mu)$, let $2 < p < \infty$, and let $S \subset \{1,\dots,n\}$ be a random set with each $i=1,\dots,n$ lying in $S$ with an independent probability of $n^{2/p-1}$.  Then with probability $\sim 1$, one has the ``$\Lambda(p)$ inequality''
$$ \| \sum_{i \in S} a_i \phi_i \|_{L^p(X,\mu)} \lesssim_p (\sum_{i \in S} |a_i|^2)^{1/2}$$
for all real or complex numbers $a_i$.
\end{theorem}

This theorem famously resolved a long-standing problem in harmonic analysis, namely whether it was possible to produce an (infinite) family of plane waves $x \mapsto e^{2\pi i nx}$ on the unit circle $\R/\Z$ which obeyed the $\Lambda(p)$ inequality but not the $\Lambda(q)$ inequality for a given choice of $2 < p < q < \infty$.

The proof of Theorem \ref{lambda-p} is quite complicated and we only give an extremely oversimplified sketch here.  The main difficulty here is one needs to control a random uncountable supremum
\begin{equation}\label{ksum}
 K \coloneqq \sup_{|a| \leq 1} \| \sum_{i=1}^n 1_{i \in S} a_i \phi_i \|_{L^p(X,\mu)} 
\end{equation}
where $a = (a_1,\dots,a_n)$ ranges over vectors of norm at most $1$. Raising this expression to the power $p$, we obtain 
\begin{equation}\label{supa}
 \sup_{|a| \leq 1} \int_X \sum_{i=1}^n 1_{i \in S} a_i \phi_i \sum_{j=1}^n 1_{j \in S} \overline{a}_j \overline{\phi}_j |\sum_{k=1}^n 1_{k \in S} a_k \phi_k|^{p-2}\ d\mu.
\end{equation}
The set $S$ appears here three times, but by randomly decomposing $S$ into three subsets $S_1,S_2,S_3$, and also judiciously splitting $a$ into three components $a,b,c$, it turns out that one can reduce the task of bounding this expression into that of bounding\footnote{This sort of trick to decouple probabilistic expressions of dependent random variables into probabilistic expressions of independent random variables is another useful member of Bourgain's toolkit that is now widely used in probability theory.  We will not discuss this decoupling trick in further detail here, but see for instance \cite{pena}.} ``decoupled'' analogues of \eqref{supa} such as
$$ \sup_{|a|,|b|,|c| \leq 1} \int_X \sum_{i=1}^n 1_{i \in S_1} a_i \phi_i \sum_{j=1}^n 1_{j \in S_2} \overline{b}_j \overline{\phi}_j |\sum_{k=1}^n 1_{k \in S_3} c_k \phi_k|^{p-2}\ d\mu.$$
In fact by using some further dyadic decompositions (in the spirit of Section \ref{dyadic-sec}) one can make further restrictions on the support and pointwise magnitudes of $a,b,c$); a typical such restriction to keep in mind is that there is some $1 \leq m \leq n$ for each of the $a,b,c$ are supported on sets of cardinality at most $m$ and are pointwise bounded by $O(m^{-1/2})$.  One then applies a variant of Dudley's inequality to control the supremum in $a$, reducing matters to controlling metric entropies of a collection of functions of the form
$$ \{ \sum_{j=1}^n 1_{j \in S_2} b_j \phi_j |\sum_{k=1}^n 1_{k \in S_3} c_k \phi_k|^{p-2}: |b|, |c| \leq 1 \}$$
(with additional constraints on the support of $b,c$ that we do not detail here).  These are controlled in turn by elementary inequalities (such as H\"older's inequality) as well as a variant of the random variable $K$ defined in \eqref{ksum}, as well as some metric entropy estimates of Bourgain, Lindenstrauss, and Milman \cite{zonotopes}.

The chaining method was later streamlined into the \emph{generic chaining} method of Talagrand, in which the metric balls $\{ t': d(t,t') \leq \eps\}$ that implicitly appear in the chaining argument are allowed to be weighted by an arbitrary measure on $T$ known as a \emph{majorizing measure}, leading to estimates that are essentially optimal in many situations, and can be used for instance to give a simplified proof of Theorem \ref{lambda-p} with a stronger conclusion, and which avoids the use of decoupling methods; see \cite{talagrand} for details.

There are many other works of Bourgain (and coauthors) in which metric entropy and chaining arguments are used to bound large deviations of supremum type quantities.  Here is a small sample:

\begin{itemize}
  \item The paper \cite{zonotopes} concerns the approximation theory of zonotopes (finite sums of intervals in a normed vector space), showing that all zonoids (limits of zonotopes) can be efficiently approximated by ``low complexity'' zonotopes.  A key step is to use a chaining argument to show that an $L^1$ norm $\|x\|_{L^1(X,\mu)}$ can be efficiently approximated by an empirical sample $\frac{1}{N_0} \sum_{i=1}^{N_0} |x(i)|$ uniformly for all $x$ in a certain convex body, relying heavily on metric entropy estimates on convex bodies such as the dual Sudakov inequality \cite{pajor}. 	
	\item In \cite{hal}, a metric entropy and concentration of measure argument is used to show that Montgomery's large values conjecture \cite{mont} in analytic number theory on the distribution of large values of a Dirichlet polynomial $\sum_{n \sim M} a_n n^{it} = \sum_{n \sim M} a_n e^{it\log n}$ is almost surely true if one replaces the frequencies $\log n$ by a suitable random set.
	\item In \cite{iso}, a metric entropy and concentration of measure argument (as well as the decoupling trick) is used to show if $K$ is a symmetric convex body $K$ of unit volume whose moment of inertia $\int_K yy^T\ dy$ is normalised to be a constant multiple $LI$ of the identity matrix, then this moment of inertia can be closely approximated by that of a surprisingly small number of randomly chosen points from $K$; this has applications to random matrix theory and statistics, in particular in allowing one to compare a covariance matrix with an empirical sample of that matrix \cite{empirical}. 
	\item In \cite{rip}, a chaining argument combined with dyadic decomposition is used to show that a randomly selected collection $S$ of columns in a bounded orthonormal system obeys the ``restricted isometry property'', of being an approximate isometry when restricted to any $m$ rows, as long as $S$ is only slightly larger than $m$, improving quantitatively over previous bounds \cite{candes,rv} in the area.
\end{itemize}

\section{Putting it all together: a case study}\label{case}

Let us say that a subset $S$ of the reals has \emph{property (E)} if, whenever every measurable subset $A$ of $\R$ of positive measure contains an affine image $x+tS \coloneqq \{ x+ty: y \in S \}$ of $S$ for some $x \in \R$ and $t \neq 0$.  An easy application of the Lebesgue density theorem shows that every finite set of reals has property (E).  The following question of Erd\H{o}s is still unsolved:

\begin{problem}[Erd\H{o}s similarity problem] Does there exist an infinite set $S$ with property (E)? 
\end{problem}

One of the strongest general negative results in this direction is by Bourgain \cite{similarity}:

\begin{theorem}[Triple sumsets fail (E), qualitative version]\label{qual}  Let $S_1,S_2,S_3$ be infinite subsets of $\R$.  Then the sumset $S_1+S_2+S_3 \coloneqq  \{s_1+s_2+s_3: s_1 \in S_1, s_2 \in S_2, s_3 \in S_3\}$ fails property (E).
\end{theorem}

The corresponding question for double sumsets $S_1+S_2$ remains open; it will be clear shortly why it is necessary in Bourgain's arguments to have at least three summands.

As we shall see, the proof of this result can largely be described as an application of several of the tools discussed in this paper.  The first step is to convert the problem to a quantitative one, as per Section \ref{quant-sec}.  The quantitative formulation is as follows:

\begin{theorem}[Triple sumsets fail (E), quantitative version]\label{quant}  Let $S_1,S_2,S_3$ be bounded infinite subsets of $\R$ containing $0$ as an adherent point.  Then there does not exist a constant $C$ for which one has the bound
$$ \int_{(\R/\Z)^J} \inf_{1 < t < 2} \sup_{x' \in x+t S_0 v} |f(x')|\ dx \leq C \int_{(\R/\Z)^J} |f(x)|\ dx$$
for all tori $(\R/\Z)^J$, all continuous $f: (\R/\Z)^J \to \R$, all vectors $v \in \R^J$, and all finite subsets $S_0$ of $S_1+S_2+S_3$.
\end{theorem}

Let us now sketch why Theorem \ref{quant} implies Theorem \ref{qual}.  (The converse implication is also true; see \cite[\S 2]{similarity}.)
Since unbounded sets clearly fail property (E), we may assume without loss of generality in Theorem \ref{qual} that $S_1,S_2,S_3$ are bounded; by Bolzano-Weierstrass and translation, we may assume that each of the $S_i$ contain $0$ as an adherent point.

Let $\delta>0$ be sufficiently small, and let $M>0$.  From Theorem \ref{quant} and rescaling, one can find a torus $(\R/\Z)^J$, a vector $v \in \R^J$, a finite subset $S_0$ of $S_1+S_2+S_3$ and a continuous function $f$ (which we can take to be non-negative) such that
$$ \int_{(\R/\Z)^J} F(x)\ dx > M \int_{(\R/\Z)^J} f(x)\ dx$$
where $F(x) \coloneqq  \inf_{\delta < t < 2\delta} \sup_{x' \in x+t S_0 v} f(x')\ dx$.  Applying dyadic pigeonholing (and writing $\int_{(\R/\Z)^J} F\ dx = \int_0^\infty |\{ F \geq \lambda \}|\ d\lambda$ and $\int_{(\R/\Z)^J} f\ dx = \int_0^\infty |\{ f > \lambda \}|\ d\lambda$), as per Section \ref{dyadic-sec}, we can then find a threshold $\lambda>0$ such that
$$ |\{ F \geq \lambda\}| > M |\{ f > \lambda\}|.$$
If we write $A \coloneqq  \{ f > \lambda\}$ and 
$$ A_1 \coloneqq \bigcap_{\delta < t < 2\delta} \bigcup_{y \in S_0} (A - tyv)$$
then we have $|A_1| \geq M |A|$.  The set $A_1$ could be small compared with $(\R/\Z)^J$; however by using the random translations trick as per Section \ref{random-sec}, we can find an open set $B \subset (\R/\Z)^J$ (a union of finitely many translates of $A$) of arbitrarily small measure such that the set
$$ B_1 \coloneqq \bigcap_{\delta < t < 2\delta} \bigcup_{y \in S_0} (B - tyv)$$
has measure arbitrarily close to $1$.  By construction, the set $B_1 \backslash B$ does not contain any set of the form $x+t(S_1+S_2+S_3)v$ with $x \in \R/\Z$ and $\delta < t < 2\delta$ (here we use the fact that $B$ is open and $0$ is an adherent point of $S_1+S_2+S_3$). If one then considers the set $\{ y \in [0,1]: x+tyv \in B_1 \backslash B\}$ for a randomly chosen $x \in (\R/\Z)^J$, one can then find a subset of $[0,1]$ of measure arbitrarily close to $1$ that does not contain any set of the form $x+tS$ with $x \in \R$ and $\delta < t < 2\delta$. Taking intersections over all small dyadic choices of $\delta$ (and then restricting to a small interval to eliminate large scales) we can then establish Theorem \ref{qual}.

Now we sketch the proof of Theorem \ref{quant}.  We have to find a continuous function $f: (\R/\Z)^J \to \R$ and a vector $v \in \R^J$ for which
$$ \int_{(\R/\Z)^J} \inf_{1 < t < 2} \sup_{x' \in x + tS_0 v} |f(x')|\ dx \gg \int_{(\R/\Z)^J} |f(x)|\ dx$$
for some finite $S_0 \subset S_1+S_2+S_3$, where we informally use $X \gg Y$ to denote the claim that $X$ is much larger than $Y$.  The next idea is to take advantage of large deviations as per Section \ref{entropy-sec}.  To do this one needs to select the vector $v$ so that the collection of dilates $tS_0 v \mod \Z^J, 1 \leq t \leq 2$ have ``low entropy''.  The construction is as follows. Let $J \coloneqq  3J_0$ for a large $J_0$, then as $S_1,S_2,S_3$ all have $0$ as an adherent point we can find non-zero real numbers $s_{i,j} \in S_i$ for $i=1,2,3$ and $j=1,\dots,J_0$ with the relative size relation
$$ |s_{1,1}| \gg \dots \gg |s_{1,J_0}| \gg |s_{2,1}| \gg \dots \gg |s_{2,J_0}| \gg |s_{3,1}| \gg \dots \gg |s_{3,J_0}| > 0.$$
We then let $v \in \R^J$ be the vector
$$ v \coloneqq \left(\frac{1}{10 s_{i,j}}\right)_{i=1,2,3; j=1,\dots,J_0}.$$
If we set $S_0 \coloneqq  \{ s_{1,j_1} + s_{2,j_2} + s_{3,j_3}: 1 \leq j_1,j_2,j_3 \leq J_0\}$, then $S_0$ is a finite subset of $S_1+S_2+S_3$.  A routine calculation shows that for any $1 \leq t \leq 2$, the set $t S_0 v \mod \Z^J$ consists of $J_0^3$ points separated from each other (in the $\ell^\infty$ metric on $(\R/\Z)^J$) by $\gtrsim 1$.  If this set had no further structure, one would then expect the ``entropy'' of such sets to be exponential in $J \times J_0^3 \sim J^4$.  However the arithmetic structure gives this set significantly lower entropy.  Indeed, observe that each of the $J$ coordinates of each of the $J_0^3$ elements of $t S_0 v \mod \Z^J$ are sums of three quantities of the form $\frac{s_{i',j'}}{10 s_{i,j}} \mod 1$ for $i,i'=1,\dots,3$, $j,j' = 1,\dots,J_0$.  Thus the set $t S_0 v \mod \Z^J$ is completely described by $O(J^2)$ parameters, so the metric entropy of these sets (using the Hausdorff metric and the $\ell^\infty$ norm on $(\R/\Z)^J$) at any given metric scale $0 < \delta < 1$ is $O(1/\delta)^{O(J^2)}$. In particular, this entropy is subexponential compared to the cardinality $J_0^3$ of $S_0$. (It is at this point that it is essential that we have at least three summands in the set $S_1+S_2+S_3$.)

Now we sketch how to construct the function $f$.  Let $\eps>0$, and suppose $J$ is large depending on $\eps$.  Partition $(\R/\Z)^J$ into cubes of sidelength (say) $\frac{1}{100 J}$, and let $E$ be the union of a random collection of these cubes, with each cube selected in $E$ with an independent probability of $\eps$.  We set $f$ to be the indicator function $1_E$ (we ignore for this sketch the requirement that $f$ be continuous, as this can be addressed by a standard mollification).  Then $\int_{(\R/\Z)^J} |f(x)|\ dx \sim \eps$ with high probability.  On the other hand, for any given $x \in (\R/\Z)^J$ and $1 < t < 2$, we will have $\sup_{x' \in x + tS_0 v} |f(x')| = 1$ with probability at least $1 - \exp( -\eps J_0^3 )$.  Using the metric entropy bound, one can then show that $\inf_{1 < t < 2} \sup_{x' \in x + tS_0 v} |f(x')| = 1$ with probability $1 - O( O(J)^{O(J^2)} \exp( -\eps J_0^3 ) )$, which is comparable to $1$ if $J$ is large enough.  This gives the claim.

\end{document}